\documentclass[11pt,a4paper]{article}%
\usepackage[centertags]{amsmath}
\usepackage{amsfonts}
\usepackage{amssymb}
\usepackage{amsthm}
\usepackage{epsfig}
\usepackage{setspace}
\usepackage{ae}
\usepackage{eucal}
\usepackage[usenames]{color}%
\usepackage{graphicx}

\theoremstyle{plain}
\newtheorem{thm}{Theorem}[section]
\newtheorem{lem}[thm]{Lemma}

\newtheorem{prop}[thm]{Proposition}
\theoremstyle{definition}

\theoremstyle{remark}

\newtheorem{rem}[thm]{Remark}

\begin{document}

\title{On a class of singular solutions to the incompressible 3-D Euler equation}
\author{J\"org Kampen }
\maketitle

\begin{abstract}
A class of singular $3D$-velocity vector fields is constructed which satisfy the incompressible $3D$-Euler equation. It is shown that such a solution scheme does not exist in dimension $2$. The solutions constructed are bounded and smooth up to finite time where they become singular. Although the solution is smooth and bounded there seems to be no bound in $L^2$ of the velocity field.
\end{abstract}

%\footnotetext[1]{Weierstrass Institute for Applied Analysis and Stochastics,
%M%ohrenstr.39, 10117 Berlin, Germany. Support by DFG Matheon is acknowledged.
%\texttt{{kampen@wias-berlin.de}}.}

\begin{center}
 
\end{center}\section{Statement of results and idea of construction of singular solutions}
The main results of this paper concern singular solutions of the $n$-D Euler equation for $n\geq 3$, i.e., the equation system
\begin{equation}\label{Eul}
\left\lbrace \begin{array}{ll}
\frac{\partial v_i}{\partial t}+\sum_{j=1}^n v_j\frac{\partial v_i}{\partial x_j}=-\frac{\partial p}{\partial x_i},\\
\\
\sum_{i=1}^n\frac{\partial v_i}{\partial x_i}=0,\\
\\
v_i(0,x)=h_i(x).
\end{array}
\right.\end{equation}
For the main part of this paper we consider evolutions on the whole domain of ${\mathbb R}^n$. Analogous results for the $n$-torus ${\mathbb T}^n$ with respect to dual Sobolev spaces may be derived, but we did not consider the details so far. It is desriable to have a construction for data of physical interest which are $C^{\infty}$  and such that all multivariate derivatives are of polynomial decay of order $m\geq 2$, i.e., for multiindices $\alpha=(\alpha_1,\cdots ,\alpha_n)$  with $\alpha_i\geq 0$ and $|\alpha|=\sum_{i=1}^n\alpha_i$ we have
\begin{equation}
D^{\alpha}_xf_i(x)\leq \frac{C_{\alpha}}{1+|x|^{|\alpha}|},
\end{equation}
where $D^{\alpha}_x\equiv\frac{\partial^{|\alpha|}}{\partial x^{\alpha}}$ denotes the multivariate derivative of order $\alpha$. In the present paper we construct a bounded smooth vector field where the first two components are of finite energy. It may be interesting to consider similar constructions for derivatives of the value functions and possibly with weaker singularities in $t$- although the connection of data and solution may not be that simple as in this paper implying that the construction has to be more involved. If a function $f$ is of polynomial decay of any order $m\in {\mathbb N}$ then we simply say that $h$ is of polynomial decay. Polynomial decay and smoothness implies that the data are in the intersection $\cap_{s\in {\mathbb R}}H^s$ of standard Sobolev spaces   $H^s\equiv H^s\left({\mathbb R}^n\right)$ of order $s\in {\mathbb R}$, and hence polynomial decay of the Fourier transforms. 
On the $n$-torus dual Sobolev spaces may be used. In this case multiindexed vectors
\begin{equation}
\mathbf{u}^F:=\left( u_{\alpha} \right)^T_{\alpha\in {\mathbb Z}^n} 
\end{equation}
(the superscript $T$ meaning 'transposed') may be considered with complex constants $u_{\alpha}$ which decay fast enough as $|\alpha|\uparrow \infty$ corresponds to a function
 \begin{equation}\label{uclass}
  u\in C^{\infty}\left({\mathbb T}^n_l\right), 
 \end{equation}
where
\begin{equation}
u(x):=\sum_{\alpha\in {\mathbb Z}^n}u_{\alpha}\exp{\left( \frac{2\pi i\alpha x}{l}\right) }.
\end{equation}
We say that the infinite vector $\mathbf{u}^F$ is in the dual Sobolev space of order $s\in {\mathbb R}$, i.e.,
\begin{equation}\label{hs}
\mathbf{u}^F\in h^s\left({\mathbb Z}^n\right) 
\end{equation}
 in symbols, if the corresponding function $u$ defined in (\ref{uclass}) is in the Sobolev space $H^s\left({\mathbb T}^n_l\right)$ of order $s\in {\mathbb R}$. We may also define the dual space $h^s$ directly via 
 \begin{equation}
 \mathbf{u}^F\in h^s\left({\mathbb Z}^n\right)\Leftrightarrow \sum_{\alpha \in{\mathbb Z}^n}|u_{\alpha}|^2\left\langle \alpha\right\rangle^{2s}< \infty, 
 \end{equation}
where
\begin{equation}
 \left\langle \alpha\right\rangle :=\left(1+|\alpha|^2 \right)^{1/2}. 
\end{equation}
The two definitions are clearly equivalent.

Our first observation reduces the search for a class of singular solutions to the search for a class of nonsingular global solutions of certain nonlinear partial integro-differential solutions. We have  
\begin{thm}\label{thm3D}
If $n=3$ then for all data $f_1,f_2\in H^2\cap C^2 $ and $f_3\in C^2$ which satisfy the nonlinear partial integro-differential equation 
\begin{equation}\label{dataeq}
\begin{array}{ll}
{\Big(} -f_1+f_1f_{1,1}+f_2f_{1,2}-\int_{{\mathbb R}^3} K_{,1}(x-y){\Big(} f_{1,1}^2+f_{2,2}^2+\left( f_{1,1}+f_{2,2}\right) ^2+\\
\\
f_{1,2}f_{2,1}+f_{1,3}I_{3,1}(f_{1,1},f_{2,2})+f_{2,3}I_{3,2}(f_{1,1},f_{2,2}){\Big)}(y)dy{\Big)}f_{2,3}\\
\\
={\Big (} -f_2+f_1f_{2,1}+f_2f_{2,2}-\int K_{,2}(x-y){\Big (} f_{1,1}^2+f_{2,2}^2+\left( f_{1,1}+f_{2,2}\right) ^2+\\
\\
f_{1,2}f_{2,1}+f_{1,3}I_{3,1}(f_{1,1},f_{2,2})+f_{2,3}I_{3,2}(f_{1,1},f_{2,2}){\Big)}(y)dy{\Big)}f_{1,3},
\end{array}
\end{equation}
then there exists a singular solution to the incompressible Euler equation which becomes singular at finite time $t_0>0$, and is smooth for $t<t_0$. Here, 
\begin{equation}
f_{3}(x)=I_3(f_{1,1},f_{2,2}):=\int_{-\infty}^{x_3}\left( -f_{1,1}-f_{2,2}\right)(x_1,x_2,y)dy 
\end{equation}
A similar result holds for $n>3$.
\end{thm}
\begin{rem}
A similar result may hold on the $n$-torus.
\end{rem}

A second observation is that solutions to (\ref{dataeq}) exist, and in a rather generic sense, i.e., in a sense which makes it impossible to single out the data set of the constructed set as 'exceptional'. We have
\begin{thm}\label{integrodiffeq}For each $f_1\in H^2\cap C^2$ which is of polynomial decay of order $2$ there exist $f_2\in H^2\cap C^2$ of polynomial decay of order $2$, and $f_3\in C^2$ such that the vector field $(f_1,f_2,f_3)^T$ is a global solution of equation (\ref{dataeq}). More general, for or each $f_1\in H^m\cap C^m$ which is of polynomial decay of order $m$ there exist $f_2\in H^m\cap C^m$ and $f_3\in C^m$ of polynomial decay of order $m$ such that the vector field $(f_1,f_2,f_3)^T$ is a global solution of equation (\ref{dataeq}).
\end{thm}
Next let us consider some background for the results above. The Euler equation is an equation without diffusion. If we add an diffusion operator $-\nu \Delta $ with a constant viscosity $\nu>0$ we get the incompressible Navier-Stokes equation
\begin{equation}\label{Nav}
\left\lbrace \begin{array}{ll}
\frac{\partial v_i}{\partial t}-\nu \Delta v_i+\sum_{j=1}^n v_j\frac{\partial v_i}{\partial x_j}=-\frac{\partial p}{\partial x_i},\\
\\
\sum_{i=1}^n\frac{\partial v_i}{\partial x_i}=0,\\
\\
v_i(0,x)=h_i(x).
\end{array}
\right.\end{equation} 
If we cancel the pressure $p$ and incompressibility we get the multivariate Burgers equation.
\begin{equation}\label{Burg}
\left\lbrace \begin{array}{ll}
\frac{\partial u_i}{\partial t}-\nu \Delta u_i+\sum_{j=1}^n u_j\frac{\partial u_i}{\partial x_j}=g_i,\\
\\
v_i(0,x)=h_i(x).
\end{array}
\right.\end{equation} 
Here the $g_i\in \cap_{s\in {\mathbb R}} H^s$ are additional source terms which may occur also in the formulation of the incompressible Navier-Stokes equation. For $\nu>0$ global existence for the multivariate Burgers equation is known. An elementary proof may be found in \cite{KB1} and \cite{KB2}. For the multivariate Burgers equation it is more or less obvious that global existence is not preserved as the equation degenerates, i.e., as $\nu \downarrow 0$. In this case, 
the function $v_i:{\mathbb R}^n\rightarrow {\mathbb R},~1\leq i\leq n$ defined by
\begin{equation}\label{solscheme}
v_i(t,x):=\frac{f_i(x)}{t-1},
\end{equation}
along with
\begin{equation}\label{spatial}
-f_i(x)+\sum_{j=1}^n f_j(x)\frac{\partial f_i(x)}{\partial x_j}=0
\end{equation}
for all $1\leq i\leq n$ and all $x\in {\mathbb R}^n$  solves the inviscid equation
\begin{equation}
\frac{\partial v_i}{\partial t}+\sum_{j=1}^n v_j\frac{\partial v_i}{\partial x_j}=0
\end{equation}
for $1\leq i\leq n$, and with initial data
\begin{equation}
v_i(0,x)=h_i(x)=:-f_i(x).
\end{equation}
Obviously, this velocity field becomes singular at $t=1$ as long as $f_i\neq 0$ for some $1\leq i\leq n$.
This is a phenomenon which we can observed in all dimensions $n\geq 1$. We may say that viscosity $\nu=0$ is a bifurcation point for existence of global behavior. Indeed for $\nu<0$ the equation is even ill-posed in most situations.   
This situation becomes more complicated and more interesting if we consider the analogous relationship between the incompressible Euler equation and the incompressible Navier-Stokes equation. For in this case we know that that the Euler equation has global solutions in dimension $n=2$ and we maintained that the incompressible Navier-Stokes equation has global solutions in different situations (cf. \cite{KB2}, \cite{K3}, and \cite{KNS}). The difference of the Euler equation and a degenerate Burgers equation is the existence of pressure and incompressibility, and the relation of both which yields elimination of the pressure by the Leray projection.
If the divergence of the vector field $\mathbf{v}=(v_1,\cdots ,v_n)^T$ of the form (\ref{solscheme}) is zero, i.e., if
\begin{equation}
\mbox{div}\mathbf{v}=\sum_{i=1}^n\frac{\partial v_i}{\partial x_i}=:\sum_{i=1}^n v_{i,i}=0,
\end{equation}
then this incompressibility together with the relation (\ref{spatial}) implies that the divergence of the vector field $\mathbf{f}=(f_1,\cdots ,f_n)^T$ is zero, and whence
\begin{equation}
\begin{array}{ll}
-\sum_{i=1}^nf_{i,i}(x)+\sum_{i=1}^n\sum_{j=1}^n \left( f_j(x)f_{i,j}(x)\right)_{,i}\\
\\
=\sum_{i,j=1}^n  f_{j,i}(x)f_{i,j}(x)=0
\end{array}
\end{equation}
However the Leray projection form of the incompressible Euler equation
is just
\begin{equation}
\left\lbrace \begin{array}{ll}
\frac{\partial v_i}{\partial t}+\sum_{j=1}^n v_j\frac{\partial v_i}{\partial x_j}=\int K_{n,i}(x-y)\frac{\partial v_i}{\partial x_j}(t,y)\frac{\partial v_j}{\partial x_i}(t,y)dy,\\
\\
v_i(0,x)=-f_i(x).
\end{array}\right.
\end{equation}
Here, $K_n$ denotes the fundamental solution of the Poisson equation in dimension $n$ and $K_{n,i}$ denotes its partial derivative with respect to the variable $x_i$.
This leads is to the conclusion
\begin{prop}
The function $v_i:{\mathbb R}^n\rightarrow {\mathbb R},~1\leq i\leq n$ defined for all $1\leq i\leq n$ by
\begin{equation}\label{velform}
v_i(t,x):=-\frac{f_i(x)}{1-t}
\end{equation}
satisfies the incompressible Euler equation with initial data $h_i=-f_i: {\mathbb R}^n\rightarrow {\mathbb R}$ if the conditions
\begin{equation}\label{spatial2}
-f_i(x)+\sum_{j=1}^n f_j(x)\frac{\partial f_i(x)}{\partial x_j}=0,~1\leq i\leq n,
\end{equation}
and 
\begin{equation}\label{spatial2}
\sum_{i,j=1}^n  f_{j,i}(x)f_{i,j}(x)=0,
\end{equation} 
and
\begin{equation}
\mbox{div}~\mathbf{f}=\sum_{i=1}^nf_{i,i}=0
\end{equation}
are satisfied.
\end{prop}
If $f_i$ satisfy periodic boundary conditions, then an analogous statement is true on the $n$-torus of course.
For a torus of size $l$ the relations (\ref{spatial}) may be written in the basis $\left\lbrace \exp\left( \frac{2\pi i kx}{l}\right) \right\rbrace_{k\in {\mathbb Z}^n}$, and where $kx=\sum k_ix_i$ along with $x=(x_1,\cdots, x_n)$. This leads to the relation
\begin{equation}
f_{i\alpha}=-\sum_{j=1}^n\sum_{\beta \in {\mathbb Z}^n}f_{j(\alpha-\beta)}\beta_jf_{i\beta}
\end{equation}
for the modes $f_{i\alpha}$ for $1\leq i\leq n$ and the multiindices $\alpha,\beta \in {\mathbb Z}^n$
Next to the two observations stated at the beginning of this section we make a third observation which we would expect from the fact that global solutions to the Euler equation exist in dimension $2$. We have
\begin{thm}\label{thm2D}
In case of dimension $n=2$ the only classical solution of form (\ref{velform}) is
\begin{equation}
f_i\equiv 0
\end{equation}
for all $1\leq i\leq n$. This is compatible with the well-known fact that classical global solutions for 2-D incompressible Euler equations exist.
\end{thm}
In the next section we prove this theorem. Then in section 3 we prove theorem \ref{thm3D} and in section 4 we prove theorem \ref{integrodiffeq}. In the last section we prove the existence of global solutions to a Cauchy problem which is related to the partial integro-differential equation of theorem \ref{thm3D}. We may summarize the considerations of this and the mentioned related papers saying that viscosity $\nu=0$ is a bifurcation point of global solution behavior in case of dimension $n\geq 3$. 

\section{Proof of theorem \ref{thm2D}}
We prove the theorem on the $n$-torus ${\mathbb T}^n$ for $n=2$.
For dimension $n=2$ divergence becomes
\begin{equation}
v_{1,1}(t,x)=-v_{2,2}(t,x)
\end{equation}
for all $(t,x)$, hence 
\begin{equation}
f_{1,1}(x)=-f_{2,2}(x).
\end{equation}
Using this incompressibility, and differentiating the relations
\begin{equation}\label{burg}
\begin{array}{ll}
v_1+v_1v_{1,1}+v_2v_{1,2}=0\\
\\
v_2+v_1v_{2,1}+v_2v_{2,2}=0
\end{array}
\end{equation}
with respect to $x_1$ and $x_2$ respectively, summing up, and multiplying by $(1-t)$ leads to
\begin{equation}
f_{1,1}^2+2f_{1,2}f_{2,1}+f_{2,2}^2=0.
\end{equation}
Along with divergence the latter equation leads to
\begin{equation}\label{frel}
f_{1,2}f_{2,1}+f_{2,2}^2=f_{1,1}^2+f_{1,2}f_{2,1}=-f_{1,1}f_{2,2}+f_{1,2}f_{2,1}=0.
\end{equation}
The last equation of (\ref{frel}) may be rewritten in the form 
\begin{equation}
\mbox{det}\left( \begin{array}{ll}f_{1,1} & f_{2,1}\\
f_{1,2}& f_{2,2}
\end{array}\right)=0. 
\end{equation}
Hence, the vectors $\left(f_{1,1},f_{1,2}\right)^T$ and  $\left(f_{2,1},f_{2,2}\right)^T$ are linearly dependent, i.e., there is a $\lambda\in {\mathbb R}$ such that
\begin{equation}\label{rel1}
f_{1,1}=\lambda f_{2,1} \mbox{ and } f_{1,2}=\lambda f_{2,2}.
\end{equation}
The latter relation may be considered together with the incompressibility relation
\begin{equation}\label{rel2}
f_{1,1}=-f_{2,2}.
\end{equation}
Consider the ansatz
\begin{equation}\label{rel3}
f_i=\sum_{\alpha \in {\mathbb Z}^2}f^i_{\alpha}\exp\left(2\pi i \alpha x\right) 
\end{equation}
for $1\leq i\leq 2$, and $\alpha=(\alpha_1,\alpha_2)\in {\mathbb Z}^2$, and $\alpha x:=\alpha_1x_1+\alpha_2x_2$. Note that (\ref{rel2}) is equivalent to the statement that for all $\alpha\in {\mathbb Z}^2$
\begin{equation}
f^1_{\alpha}\alpha_1=-f^2_{\alpha}\alpha_2.
\end{equation}
This implies that $\lambda=-1$ in (\ref{rel1}). Now from (\ref{burg}) we have
\begin{equation}
\begin{array}{ll}
-f_1+f_1f_{1,1}+f_2f_{1,2}=0\\
\\
-f_2+f_1f_{2,1}+f_2f_{2,2}=0,
\end{array}
\end{equation}
and this together with the information $\lambda=-1$ in (\ref{rel1}) and incompressibility leads to
\begin{equation}\label{rel4}
\begin{array}{ll}
-f_1+f_1f_{1,1}+f_2f_{1,2}=-f_1+f_1f_{1,1}-f_2f_{2,2}=0,\\
\\
-f_2+f_1f_{2,1}+f_2v_{2,2}=-f_2-f_1f_{1,1}+f_2f_{2,2}=0,
\end{array}
\end{equation}
Adding both equations we get
\begin{equation}
f_2=-f_1.
\end{equation}
This together with (\ref{rel4}) and vanishing divergence leads to
\begin{equation}
\begin{array}{ll}
-f_1+f_1f_{1,1}-f_2f_{2,2}=-f_1+f_1(f_{1,1}+f_{2,2})=-f_1=0,\\
\\
-f_2-f_1f_{1,1}+f_2f_{2,2}=-f_2+f_2(f_{1,1}+f_{2,2})=-f_2=0,
\end{array}
\end{equation}
and we are done.
\section{Proof of theorem \ref{thm3D}}
Let us first explain why the theorem is stated with finite (maybe smaller) $t_0>0$ instead of $t_0=1$ as may be expected from reading the considerations made so far. Indeed, in the following we define
\begin{equation}
v_i(t,x):=\frac{f_i(x)}{t-1}
\end{equation}
and derive an equation as in the statement of theorem \ref{thm3D} above. However this imposes the condition
\begin{equation}\label{below}
f_{1,1}+f_{2,2}>-1
\end{equation}
which is satisfied for sufficiently small $H^1$-norm of $f_1$ and $f_2$. However reading the argument below with the vector field 
\begin{equation}
\tilde{v}_i(t,x):=\frac{f_i(x)}{ct-1},~1\leq i\leq n,~ c>0
\end{equation}
leads to the concusion that the restriction in (\ref{below}) is no essential restriction.
If we choose such a vector field, then the argument below imposes the condition 
\begin{equation}
f_{1,1}+f_{2,2}>-c,
\end{equation}
and choosing $c>0$ large enough this is no essential restriction: given some $f_1\in H^2\cap C^2$ we shall show that there are $f_2\in H^2\cap C^2$ and $f_3\in C^2$ such that we have a divergence free vector field, and such that the pair $(f_1,f_2)$ solves the partial integro-differential equation of theorem \ref{thm3D}. We may then choose $c>0$ as an upper bound for $|f_{1,1}|+|f_{2,2}|$ and apply the following argument in order to show that for this data we have a solution which becomes singular for some $t_0<\frac{1}{c}$. The reader should think of small $H^1$-data first such that the condition in (\ref{below}) is satisfied and then consider the preceding remark in order to convince himself that the condition (\ref{below}) is a convenient assumption which can be easily dropped.
We prove the theorem for $\mathbb{R}^3$ and make some remarks that it applies also to the case where the domain is the $n$-torus for $n=3$ as we go along. We construct a vector field $(f_1,f_2,f_3)^T$
with functions $f_i:C^{\infty},~1\leq i\leq n$, and which satisfies the following three conditions
\begin{itemize}
 \item[i)] The data satisfy a multivariate Burgers equation
 \begin{equation}
\begin{array}{ll}
-f_1+f_1f_{1,1}+f_2f_{1,2}+f_3f_{1,3}=G_1\\
  \\
-f_2+f_1f_{2,1}+f_2f_{2,2}+f_3f_{2,3}=G_2\\
\\
-f_3+f_1f_{3,1}+f_2f_{3,2}+f_3f_{3,3}=G_3
\end{array}
\end{equation}
where $G_i\in C^2$ are some source term functions chosen below. 
\\
\\
\item[ii)] The divergence of the vector field is zero, i.e.,
           \begin{equation}
           f_{1,1}+f_{2,2}+f_{3,3}=0
           \end{equation}
\item[iii)]
The Leray projection term of the vector field $\left(f_1,f_2,f_3\right)^T$ satisfies
\begin{equation}
          f_{1,1}^2+f_{2,2}^2+f_{3,3}^2+f_{1,2}f_{2,1}+f_{1,3}f_{3,1}+f_{2,3}f_{3,2}=g
         \end{equation}
for some function $g\in C^{2}$.
In case of the domain ${\mathbb R}^n$ the condition takes the form
\begin{equation}
G_i(x)=\int K_{,i}(x-y)g(y)dy, 
\end{equation}
where $K$ is the Poisson kernel in dimension $N=3$ and $K_{,i}$ denotes its partial derivative with respect to the variable $x_i$. 
Furthermore in case of the $n$-torus $g$ is such that
the functions $G_i$ above satisfy for $1\leq i\leq n$
\begin{equation}
G_{i}=\frac{\sum_{\alpha \in {\mathbb Z}^3}\left( \frac{2\pi i\alpha_i}{l}\right) g_{\alpha}\exp\left(\frac{2\pi i \alpha x}{l} \right) }{\sum_{i=1}^n\alpha_i^2},
\end{equation}
where 
\begin{equation}
g(x)=\sum_{\alpha \in {\mathbb Z}^3}g_{\alpha}\exp\left(\frac{2\pi i \alpha x}{l} \right) 
\end{equation}
(In the following we consider the case $l=1$ w.l.o.g..)
\end{itemize}
It is clear then that the velocity field
\begin{equation}
v_i(t,x)=\frac{f_i(x)}{t-1},~1\leq i\leq n
\end{equation}
satisfies the incompressible Navier Stokes equation in its Leray projection form with initial data $(f_1(x),f_2(x),f_3(x))^T$ on the domain $[0,1)\times {\mathbb R}^n$ (and $[0,1)\times {\mathbb T}^n$ respectively). Equivalently, it satisfies a multivariate Burgers equation
\begin{equation}
v_{i,t}+v_1v_{i,1}+v_2v_{i,2}+v_3v_{i,3}=\frac{G_i}{(1-t)^2},
\end{equation}
for $1\leq i\leq n$ by construction. We want to show that for 
$f_1\in H^2\cap C^2$  
there exists $f_2\in H^2\cap C^2$ and $f_3\in C^2$ such that conditions i),ii), and iii) are satisfied. 
First for any data $f_1,f_2\in H^2\cap C^2$ we have decay to zero at infinity for derivatives up to second order, i.e., as $|x|\uparrow \infty$. Hence from incompressibility
(condition ii) above) we have
\begin{equation}
f_{3}(x)=I_3(f_{1,1},f_{2,2}):=\int_{-\infty}^{x_3}\left( -f_{1,1}-f_{2,2}\right)(x_1,x_2,y)dy 
\end{equation}
(we may invoke the assumption of some polynomial decay here but it is not necessary).
 Next we reduce condition i). The third equation of condition i) together with incompressibility $f_{3,3}=-f_{1,1}-f_{2,2}$ leads to
\begin{equation}
-f_3+f_{1}f_{3,1}+f_2f_{3,2}+f_3\left(-f_{1,1}-f_{2,2}\right)=G_3. 
\end{equation}
This means that
\begin{equation}
f_3=-\frac{G_3-f_1f_{3,1}-f_2f_{3,2}}{-1-f_{1,1}-f_{2,2}},
\end{equation}
where $f_{1,1}+f_{2,2}\neq 1$ and
\begin{equation}
f_1f_{3,1}-f_2f_{3,2}=G_3
\end{equation}
otherwise. For the moment we assume that $1$ is not in the range of $f_{1,1}+f_{2,2}$ such that the first alternative holds for all arguments $x\in {\mathbb T}^n$.
Then condition i) reduces to 
\begin{equation}\label{burgersred}
\begin{array}{ll}
-f_1+f_1f_{1,1}+f_2f_{1,2}+\left(\frac{G_3-f_1f_{3,1}-f_2f_{3,2}}{1-f_{1,1}-f_{2,2}}\right) f_{1,3}=G_1\\
  \\
-f_2+f_1f_{2,1}+f_2f_{2,2}+\left(\frac{G_3-f_1f_{3,1}-f_2f_{3,2}}{1-f_{1,1}-f_{2,2}}\right)f_{2,3}=G_2
\end{array}
\end{equation}
We have not defined $G_1$, $G_2$, and $G_3$ yet. They are defined via the condition iii). We are free to define $g$ in this condition, i.e., we {\it define}
\begin{equation}
\begin{array}{ll}
g:= f_{1,1}^2+f_{2,2}^2+f_{3,3}^2+f_{1,2}f_{2,1}+f_{1,3}f_{3,1}+f_{2,3}f_{3,2}\\
\\
=f_{1,1}^2+f_{2,2}^2+\left( f_{1,1}+f_{2,2}\right) ^2+f_{1,2}f_{2,1}+f_{1,3}f_{3,1}+f_{2,3}f_{3,2}
\end{array}
\end{equation}
The system (\ref{burgersred}) is linear in $f_{3,1}$ and $f_{3,2}$. We get 
\begin{equation}\label{burgersred1}
\begin{array}{ll}
\left( 1-f_{1,1}-f_{2,2}\right) \left( -f_1+f_1f_{1,1}+f_2f_{1,2}-G_1\right) =\left(-G_3+f_1f_{3,1}+f_2f_{3,2}\right) f_{1,3}\\
  \\
\left(1-f_{1,1}-f_{2,2}\right)\left( -f_2+f_1f_{2,1}+f_2f_{2,2}-G_2\right) =\left(-G_3+f_1f_{3,1}+f_2f_{3,2}\right)f_{2,3}.
\end{array}
\end{equation}
The first equation multiplied by $f_{2,3}$ has to equal the second equation multiplied by $f_{1,3}$
\begin{equation}
\begin{array}{ll}
\left( 1-f_{1,1}-f_{2,2}\right){\big [} \left( -f_1+f_1f_{1,1}+f_2f_{1,2}-G_1\right)f_{2,3}\\
\\
-\left( -f_2+f_1f_{2,1}+f_2f_{2,2}-G_2\right)f_{1,3} {\big ]}=0,
\end{array}
\end{equation}
and according to our temporary assumption that $1$ is not in the range of $f_{1,1}+f_{2,2}$ this means that
\begin{equation}\label{fcon}
\begin{array}{ll}
\left( -f_1+f_1f_{1,1}+f_2f_{1,2}-G_1\right)f_{2,3}
=\left( -f_2+f_1f_{2,1}+f_2f_{2,2}-G_2\right)f_{1,3}.
\end{array}
\end{equation}
In the latter equation $G_i,~1\leq i\leq 2$ are functionals of $f_1$ and $f_2$. 
\begin{equation}
\begin{array}{ll}
G_i(x)=\int K_{,i}(x-y)g(y)dy\\
\\
=\int_{{\mathbb R}^3} K_{,i}(x-y)\times\\
\\
\left( f_{1,1}^2+f_{2,2}^2+\left( f_{1,1}+f_{2,2}\right) ^2+f_{1,2}f_{2,1}+f_{1,3}I_{3,1}(f_{1,1},f_{2,2})+f_{2,3}I_{3,2}(f_{1,1},f_{2,2})\right)(y)dy
\end{array}
\end{equation}
Hence we get the integro-differential equation
\begin{equation}\label{fcon2}
\begin{array}{ll}
{\Big(} -f_1+f_1f_{1,1}+f_2f_{1,2}-\int_{{\mathbb R}^3} K_{,1}(x-y){\Big(} f_{1,1}^2+f_{2,2}^2+\left( f_{1,1}+f_{2,2}\right) ^2+\\
\\
f_{1,2}f_{2,1}+f_{1,3}I_{3,1}(f_{1,1},f_{2,2})+f_{2,3}I_{3,2}(f_{1,1},f_{2,2}){\Big)}(y)dy{\Big)}f_{2,3}\\
\\
={\Big (} -f_2+f_1f_{2,1}+f_2f_{2,2}-\int K_{,2}(x-y){\Big (} f_{1,1}^2+f_{2,2}^2+\left( f_{1,1}+f_{2,2}\right) ^2+\\
\\
f_{1,2}f_{2,1}+f_{1,3}I_{3,1}(f_{1,1},f_{2,2})+f_{2,3}I_{3,2}(f_{1,1},f_{2,2}){\Big)}(y)dy{\Big)}f_{1,3}.
\end{array}
\end{equation}

On the $3$ torus a similar expression for  $G_i$ for $l=1$ may be defined in terms of $g$ by
\begin{equation}
G_{i}=\frac{\sum_{\alpha \in {\mathbb Z}^3}\left( (2\pi i\alpha_i)\right) g_{\alpha}\exp\left(2\pi i \alpha x\right) }{\sum_{i=1}^n\alpha_i^2},
\end{equation}
and together with the representations
\begin{equation}
f_1(x)=\sum_{\alpha \in {\mathbb Z}^3}f^1_{\alpha}\exp\left(2\pi i \alpha x \right), 
\end{equation}
and
\begin{equation}
f_2(x)=\sum_{\alpha \in {\mathbb Z}^3}f^2_{\alpha}\exp\left(2\pi i \alpha x \right).
\end{equation}

 \section{Proof of theorem \ref{integrodiffeq}}
% \begin{equation}\label{dataeqcom}
% \begin{array}{ll}
% {\Big(} -f_1+f_1f_{1,1}+f_2f_{1,2}-\int_{{\mathbb R}^3} K_{,1}(x-y){\Big(} f_{1,1}^2+f_{2,2}^2+\left( f_{1,1}+f_{2,2}\right) ^2+\\
% \\
% f_{1,2}f_{2,1}+f_{1,3}I_{3,1}(f_{1,1},f_{2,2})+f_{2,3}I_{3,2}(f_{1,1},f_{2,2}){\Big)}(y)dy{\Big)}f_{2,3}\\
% \\
% ={\Big (} -f_2+f_1f_{2,1}+f_2f_{2,2}-\int K_{,2}(x-y){\Big (} f_{1,1}^2+f_{2,2}^2+\left( f_{1,1}+f_{2,2}\right) ^2+\\
% \\
% f_{1,2}f_{2,1}+f_{1,3}I_{3,1}(f_{1,1},f_{2,2})+f_{2,3}I_{3,2}(f_{1,1},f_{2,2}){\Big)}(y)dy{\Big)}f_{1,3}.
% \end{array}
% \end{equation}
% 
% % -f_1+f_1f_{1,1}+f_2f_{1,2}
We rewrite (\ref{dataeq}) as
 \begin{equation}\label{dataeq2}
\begin{array}{ll}
 (-f_1+f_1f_{1,1}+f_2f_{1,2})f_{2,3}+(f_2-f_1f_{2,1}-f_2f_{2,2})f_{1,3}\\
\\
={\Big (}\int_{{\mathbb R}^3} K_{,1}(x-y){\Big(} f_{1,1}^2+f_{2,2}^2+\left( f_{1,1}+f_{2,2}\right) ^2+\\
\\
f_{1,2}f_{2,1}+f_{1,3}I_{3,1}(f_{1,1},f_{2,2})+f_{2,3}I_{3,2}(f_{1,1},f_{2,2}){\Big)}(y)dy{\Big)}f_{2,3}\\
\\
-{\Big (} \int K_{,2}(x-y){\Big (} f_{1,1}^2+f_{2,2}^2+\left( f_{1,1}+f_{2,2}\right) ^2+\\
\\
f_{1,2}f_{2,1}+f_{1,3}I_{3,1}(f_{1,1},f_{2,2})+f_{2,3}I_{3,2}(f_{1,1},f_{2,2}){\Big)}(y)dy{\Big)}f_{1,3}.
\end{array}
\end{equation}
% Since both integrals are convolutions we may rewrite this again in the form
% \begin{equation}\label{dataeq3}
% \begin{array}{ll}
% {\Big(} f_1-f_2+f_1(f_{1,1}-f_{2,1})+f_2(f_{1,2}-f_{2,2})=\\
% \\
% \int_{{\mathbb R}^3} K(x-y){\Bigg(} f_{1,1}^2+f_{2,2}^2+\left( f_{1,1}+f_{2,2}\right) ^2+\\
% \\
% f_{1,2}f_{2,1}+f_{1,3}I_{3,1}(f_{1,1},f_{2,2})+f_{2,3}I_{3,2}(f_{1,1},f_{2,2}){\Big)}_{,1}(y)dy{\Big)}f_{2,3}\\
% \\
% -{\Big (} f_{1,1}^2+f_{2,2}^2+\left( f_{1,1}+f_{2,2}\right) ^2+\\
% \\
% f_{1,2}f_{2,1}+f_{1,3}I_{3,1}(f_{1,1},f_{2,2})+f_{2,3}I_{3,2}(f_{1,1},f_{2,2}){\Big)}_{,2}(y)dy{\Big)}{\Bigg )}f_{1,3}.
% \end{array}
% \end{equation}
In the following discussion we consider the case $n=3$ and work on the whole domain ${\mathbb R}^3$. Similar considerations hold for the $n$-torus.
% In (\ref{dataeq2}) the first order derivatives of the Laplacian kernel function $K(.)$ are not square integrable near the origin, but they are in $L^2\left({\mathbb R}^n\setminus B_1(0) \right)$, where $B_1(0)$ denotes the ball of radius $1$ around the origin. A simple idea of an iteration scheme may be the following.
Given $f_1\in H^2\cap C^2$ you may try to construct $f_2$ as a limit $f_2=\lim_{k\uparrow \infty}f^k_2$ where for a start value $f^0_2=-f_1$ you may try 
%consider the telescop sums $f^k_1=f^0_1+\sum_{m=1}^{k}(f^m_1-f^{m-1}_1)$,  
the iteration scheme 
\begin{equation}\label{dataeq4}
\begin{array}{ll}
{\Big(} (-f_1+f_1f_{1,1}+f_2^{k}f_{1,2})f_{2,3}^{k}+(f_2^{k+1}-f_1f^k_{2,1}-f^{k}_2f^k_{2,2})f_{1,3}=\\
\\
\int_{{\mathbb R}^3} ((1-\phi_1)K)_{,1}(x-y){\Big(} (f_{1,1})^2+(f_{2,2}^{k})^2+\left( f_{1,1}+f^k_{2,2}\right) ^2+\\
\\
f_{1,2}f^k_{2,1}+f_{1,3}I_{3,1}(f_{1,1},f^k_{2,2})+f^k_{2,3}I_{3,2}(f_{1,1},f^k_{2,2}){\Big)}(y)dy{\Big)}f^k_{2,3}\\
\\
-{\Big (} \int K_{,2}(x-y){\Big (} f_{1,1}^2+(f_{2,2}^k)^2+\left( f_{1,1}+f^k_{2,2}\right) ^2+\\
\\
f_{1,2}f^k_{2,1}+f_{1,3}I_{3,1}(f_{1,1},f^k_{2,2})+f^k_{2,3}I_{3,2}(f_{1,1},f^k_{2,2}){\Big)}(y)dy{\Big)}f_{1,3}.
\end{array}
\end{equation}
The problem with this is that it is not easy to find an appropriate function space for the iteration, i.e., if you start an iteration step with $f^k_2\in H^2$, then it seems difficult to observe that $f^{k+1}_2\in H^2$. For this reason we consider the following Cauchy problem:
\begin{equation}\label{Cauchyepsilon}
\left\lbrace \begin{array}{ll}
\frac{\partial v^{\epsilon}}{\partial t}-\epsilon \Delta v^{\epsilon}+{\Big(} 
(-f_1+f_1f_{1,1}+v^{\epsilon}f_{1,2})v^{\epsilon}_{,3}+(v^{\epsilon}-f_1v^{\epsilon}_{,1}-v^{\epsilon}v^{\epsilon}_{,2})f_{1,3}{\Big )}=\\
\\
\int_{{\mathbb R}^3} K_{,1}(x-y){\Big(} f_{1,1}^2+\left( v^{\epsilon}_{,2}\right) ^2+\left( f_{1,1}+v^{\epsilon}_{,2}\right) ^2+\\
\\
f_{1,2}v^{\epsilon}_{,1}+f_{1,3}I_{3,1}(f_{1,1},v^{\epsilon}_{,2})+f_{2,3}I_{3,2}(f_{1,1},v^{\epsilon}_{,2}){\Big)}(y)dy{\Big)}v^{\epsilon}_{,3}\\
\\
-{\Big (} \int_{{\mathbb R}^3} K_{,2}(x-y){\Big (} f_{1,1}^2+\left( v^{\epsilon}_{,2}\right) ^2+\left( f_{1,1}+v^{\epsilon}_{,2}\right) ^2+\\
\\
f_{1,2}v^{\epsilon}_{2,1}+f_{1,3}I_{3,1}(f_{1,1},v^{\epsilon}_{2,2})+v^{\epsilon}_{,3}I_{3,2}(f_{1,1},v^{\epsilon}_{,2}){\Big)}(y)dy{\Big)}f_{1,3},\\
\\
v^{\epsilon}(0,x)=0.
\end{array}
\right.\end{equation}

\begin{thm}
The function $v^0_{\infty}:{\mathbb R}^n\rightarrow {\mathbb R}$ defined for all $x\in {\mathbb R}^n$ by
\begin{equation}
v^{0}_{\infty}(x):=\lim_{\epsilon\downarrow 0}\lim_{t\uparrow \infty}v^{\epsilon}(t,x)
\end{equation}
is well defined, and such that $v^0_{\infty}\in H^2\cap C^2$.
\end{thm}
\begin{proof}
For $\epsilon>0$ let $v^{\epsilon}$ be the solution of the Cauchy problem (\ref{Cauchyepsilon}) which we construct in the following section. In terms of the fundamental solution of 
\begin{equation}\label{fund}
\frac{\partial v^{\epsilon}}{\partial t}-\epsilon \Delta v^{\epsilon}=0,
\end{equation}
which is (for $t>s$)
\begin{equation}
\frac{1}{(2\sqrt{\epsilon \pi (t-s)})^n}\exp\left(-\frac{(x-y)^2}{4\epsilon (t-s) } \right),
\end{equation}
the solution $v^{\epsilon}$ of the Cauchy problem (\ref{Cauchyepsilon}) has the pointwise representation (with $s'=(t-s)$, therefore an additional minus sign)
\begin{equation}\label{fundrep}
\begin{array}{ll}
v^{\epsilon}(t,x)=\int_0^t\int_{{\mathbb R}^3}{\big (}
(-f_1+f_1f_{1,1}+v^{\epsilon}f_{1,2})v^{\epsilon}_{,3}+(v^{\epsilon}-f_1v^{\epsilon}_{,1}-v^{\epsilon}v^{\epsilon}_{,2})f_{1,3}{\big )}\\
\\
\times\frac{1}{(2\sqrt{\epsilon \pi s'})^n}\exp\left(-\frac{(x-y)^2}{4\epsilon s' } \right) dzds'\\
\\
+\int_0^t{\Big (}\int_{{\mathbb R}^3} K_{,1}(z-y){\Big(} f_{1,1}^2+\left( v^{\epsilon}_{,2}\right) ^2+\left( f_{1,1}+v^{\epsilon}_{,2}\right) ^2+\\
\\
f_{1,2}v^{\epsilon}_{,1}+f_{1,3}I_{3,1}(f_{1,1},v^{\epsilon}_{,2})+f_{2,3}I_{3,2}(f_{1,1},v^{\epsilon}_{,2}){\Big)}(y)dy{\Big)}v^{\epsilon}_{,3}{\Big )}(s',z)\\
\\
\times \frac{1}{(2\sqrt{\epsilon \pi s'})^n}\exp\left(-\frac{(x-y)^2}{4\epsilon s' } \right) dzds'\\
\\
-\int_0^t{\Big (} \int_{{\mathbb R}^3} K_{,2}(x-y){\Big (} f_{1,1}^2+\left( v^{\epsilon}_{,2}\right) ^2+\left( f_{1,1}+v^{\epsilon}_{,2}\right) ^2+\\
\\
f_{1,2}v^{\epsilon}_{2,1}+f_{1,3}I_{3,1}(f_{1,1},v^{\epsilon}_{2,2})+v^{\epsilon}_{,3}I_{3,2}(f_{1,1},v^{\epsilon}_{,2}){\Big)}(y)dyf_{1,3}{\Big)}(s,z)\\
\\
\times \frac{1}{(2\sqrt{\epsilon\pi s'})^n}\exp\left(-\frac{(x-y)^2}{4\epsilon s' } \right) dzds.
\end{array}
\end{equation}
The time derivative is the evaluation of the integrand on the right side of (\ref{fundrep}) at time $t$. Note that this integrand evaluated at $t$ has an upper bound
\begin{equation}
\frac{C}{(2\sqrt{\epsilon\pi t})^n}
\end{equation}
which goes to zero as $t\uparrow \infty$.
 It follows that
\begin{equation}
\lim_{t\uparrow \infty}\frac{\partial v^{\epsilon}}{\partial t}\equiv 0.
\end{equation}
Furthermore, 
\begin{equation}
v^{\epsilon}_{\infty}:=\lim_{t\uparrow \infty}v^{\epsilon}(t,.)
\end{equation}
is well defined, and we have $v^{\epsilon}_{\infty}\in H^2\cap C^2$ according to the result on the Cauchy problem of the next section. Moreover $v^{\epsilon}_{\infty}$ solves
\begin{equation}\label{elliptic}
\left\lbrace \begin{array}{ll}
-\epsilon \Delta v^{\epsilon}+{\Big(}(-f_1+f_1f_{1,1}+v^{\epsilon}f_{1,2})v^{\epsilon}_{,3}+(v^{\epsilon}-f_1v^{\epsilon}_{,1}-v^{\epsilon}v^{\epsilon}_{,2})f_{1,3}{\Big )}=\\
\\
\int_{{\mathbb R}^3} K_{,1}(x-y){\Big(} f_{1,1}^2+\left( v^{\epsilon}_{,2}\right) ^2+\left( f_{1,1}+v^{\epsilon}_{,2}\right) ^2+\\
\\
f_{1,2}v^{\epsilon}_{,1}+f_{1,3}I_{3,1}(f_{1,1},v^{\epsilon}_{,2})+f_{2,3}I_{3,2}(f_{1,1},v^{\epsilon}_{,2}){\Big)}(y)dy{\Big)}v^{\epsilon}_{,3}\\
\\
-{\Big (} \int_{{\mathbb R}^3} K_{,2}(x-y){\Big (} f_{1,1}^2+\left( v^{\epsilon}_{,2}\right) ^2+\left( f_{1,1}+v^{\epsilon}_{,2}\right) ^2+\\
\\
f_{1,2}v^{\epsilon}_{2,1}+f_{1,3}I_{3,1}(f_{1,1},v^{\epsilon}_{2,2})+v^{\epsilon}_{,3}I_{3,2}(f_{1,1},v^{\epsilon}_{,2}){\Big)}(y)dy{\Big)}f_{1,3}.
\end{array}
\right.\end{equation}
This holds for all $\epsilon>0$ where $|v^{\epsilon}|_{H^2}\cap C^2\leq C$ for some finite $C>0$ which is independent of $\epsilon>0$ (cf. next section). Hence the limit $f_2=\lim_{\epsilon \downarrow 0}v^{\epsilon}_{\infty}$ exists and solves the equation (\ref{elliptic}) for $\epsilon \downarrow 0$. 
\end{proof}

\section{Proof of global existence for the Cauchy problem}
We prove global existence via an iteration. There are several possibilities which part of the iteration of the linearization is put into a source term and which part is taken into the definition of the construction of the fundamental solution at each iteration step. In \cite{KNS} and in \cite{KB1} and \cite{KB2} we used the adjoint of fundamental solutions in order to estimate the local growth of such kind of iterations in a $H^2$-norm. In this paper we simplify the argument a little and avoid the adjoint and some estimates, i.e, we represent iterative approximations directly by convolutions with a Gaussian. An additional difficulty (compared to the schemes of the incompressible Navier Stokes equation) are the integral terms involving the operators $I_{3,1}$ and $I_{3,2}$. It seems that a contraction property for the functional increments of our scheme cannot be proved directly. However, these terms just ensure incompressibility of the vector field. Therefore we replace these integrals terms first by a parameter function $w(t,.)\in C^2_b$ and then solve the Cauchy problem for a fixed parameter $w$. Then we set up a second integral equation corresponding to incompressibility, and the solution to the latter equation is the solution leads to the solution of the Cauchy problem. 
 We start with $v^{\epsilon,0}=-f_1$, and given $v^{\epsilon,w,k}$ we define $v^{\epsilon,w,k+1}$ by the iteration
\begin{equation}\label{Cauchyit}
\begin{array}{ll}
v^{\epsilon,w,k+1}(t,x)=\int_0^t\int_{{\mathbb R}^3}{\big (} {\Big(}(-f_1+f_1f_{1,1}+v^{\epsilon,w,k}f_{1,2})v^{\epsilon,w,k}_{,3}\\
\\
+(v^{\epsilon,w,k}-f_1v^{\epsilon,w,k}_{,1}-v^{\epsilon,w,k}v^{\epsilon,w,k}_{,2})f_{1,3}{\Big )}G_{\epsilon}(t-s,x-y) dyds\\
\\
+\int_0^t{\Big (}\int_{{\mathbb R}^3} K_{,1}(z-y){\Big(} f_{1,1}^2+\left( v^{\epsilon,w,k}_{,2}\right) ^2+\left( f_{1,1}+v^{\epsilon,w,k}_{,2}\right) ^2+\\
\\
f_{1,2}v^{\epsilon,w,k}_{,1}+f_{1,3}w_{,1}+v^{\epsilon,w,k}_{,3}w_{,2}{\Big)}(y)dy{\Big)}v^{\epsilon,w,k}_{,3}{\Big )}(s,z) G_{\epsilon}(t-s,x-z)dzds\\
\\
-\int_0^t{\Big (} \int_{{\mathbb R}^3} K_{,2}(x-y){\Big (} f_{1,1}^2+\left( v^{\epsilon,w,k}_{,2}\right) ^2+\left( f_{1,1}+v^{\epsilon,w,k}_{,2}\right) ^2+\\
\\
f_{1,2}v^{\epsilon,w,k}_{2,1}+f_{1,3}w_{,1}+v^{\epsilon,w,k}_{,3}w_{,2}{\Big)}(y)dyf_{1,3}{\Big)}(s,z) G_{\epsilon}(t-s,x-z)dzds.
\end{array}
\end{equation}
This is the solution scheme for a Cauchy problem parameterized by $w(t,.)\in C^2_b$, where the corresponding naive iteration scheme would be one without diffusion.
Note that $G_{\epsilon}(t-s,x-y)=\frac{1}{(2\sqrt{\epsilon \pi (t-s)})^n}\exp\left(-\frac{(x-y)^2}{4\epsilon (t-s) } \right)$ is the fundamental solution of
\begin{equation}
\frac{\partial p}{\partial t}-\epsilon \Delta p =0.
\end{equation}
For $u:[0,\infty)\times {\mathbb R}^n\rightarrow {\mathbb R}$ we define
\begin{equation}\label{expnorm}
{\big |}u{\big |}^{\mbox{exp},C}_{H^2}:=\sup_{t\in [0,\infty)}\exp\left(-Ct\right) {\big |}u(t,.){\big |}_{H^2},
\end{equation}
where $C>0$ is determined below in order to obtain a global contraction. Next we prove a contraction property for the parameterized Cauchy problem, i.e., we prove for given $w(t,.)\in  C^2$
\begin{equation}
{\big |} \delta v^{\epsilon ,w,k+1}{\big |}^{\mbox{exp},C}_{H^2}\leq \frac{1}{2}{\big |} \delta v^{\epsilon ,w,k}{\big |}^{\mbox{exp},C}_{H^2}.
\end{equation}
In the following we denote the space associated with the norm (\ref{expnorm}) with a specific $C>0$ by $H^2_{\mbox{exp},C}$.
Next for the functional increments $\delta v^{\epsilon,w,k+1}=v^{\epsilon,k+1}-v^{\epsilon,w,k}$ we have the Cauchy problem

\begin{equation}\label{Cauchyepsilondelta}
\left\lbrace \begin{array}{ll}
\frac{\partial \delta v^{\epsilon ,w,k+1}}{\partial t}-\epsilon \Delta \delta v^{\epsilon,w,k+1}\\
\\
= {\Big(}(-f_1+f_1f_{1,1}+v^{\epsilon,w,k}f_{1,2})\delta v^{\epsilon,w,k}_{,3}
+\delta v^{\epsilon,w,k}f_{1,2} v^{\epsilon,w,k}_{,3}\\
\\
+(\delta v^{\epsilon,w,k}-f_1\delta v^{\epsilon,w,k}_{,1}-v^{\epsilon,w,k}\delta v^{\epsilon,w,k}_{,2}-\delta v^{\epsilon,w,k}v^{\epsilon,w,k}_{,2})f_{1,3}{\Big )}\\
\\
+S^{\epsilon,k+1}_w-S^{\epsilon,k}_w,\\
\\
\delta v^{\epsilon,w,k+1}(0,x)=0,
\end{array}\right.
\end{equation}
where
\begin{equation}
\begin{array}{ll}
 S^{\epsilon,k+1}_w\equiv \int_{{\mathbb R}^3} K_{,1}(x-y){\Big(} f_{1,1}^2+\left( v^{\epsilon ,w,k}_{,2}\right) ^2+\left( f_{1,1}+v^{\epsilon ,w,k}_{,2}\right) ^2+\\
\\
f_{1,2}v^{\epsilon,w,k}_{,1}+f_{1,3}w_{,1}+v^{\epsilon ,w,k}_{,3}w_{,2}{\Big)}(y)dy{\Big)}v^{\epsilon,w,k}_{,3}\\
\\
-{\Big (} \int_{{\mathbb R}^3} K_{,2}(x-y){\Big (} f_{1,1}^2+\left( v^{\epsilon ,w,k}_{,2}\right) ^2+\left( f_{1,1}+v^{\epsilon ,w,k}_{,2}\right) ^2+\\
\\
f_{1,2}v^{\epsilon ,w,k}_{,1}+f_{1,3}w_{,1}+v^{\epsilon ,w,k}_{,3}w_{,2}{\Big)}(y)dy{\Big)}f_{1,3}.
\end{array}
\end{equation}
Note that the source term $S^{\epsilon,k+1}$ for the original Cauchy problem is obtained by replacing $w_{,2}$ by $I_{3,2}(f_{1,1},v^{\epsilon ,k}_{,2})$ and $w_1$ by  $I_{3,1}(f_{1,1},v^{\epsilon ,k}_{,2})$ in $S^{\epsilon,k+1}_w$.
We shall construct the functions $v^{k,w,\epsilon}(t,.)\in H^2\cap C^2$ uniformly with respect to $t\geq 0$. The limit of the iteration is then in $H^2\subset C^{\alpha}$ uniformly in $t$, and this leads to representations of the value function $v^{\epsilon}$ which imply more regularity.
First note that we can represent the solution $\delta v^{\epsilon ,w,k+1}$ of the Cauchy problem (\ref{Cauchyepsilondelta}) in the form
\begin{equation}\label{deltavrep}
\begin{array}{ll}
\delta v^{\epsilon,w,k+1}=\int_0^t\int_{{\mathbb R}^3}{\Big(}(-f_1+f_1f_{1,1}+v^{\epsilon,w,k}f_{1,2})\delta v^{\epsilon,w,k}_{,3}
+\delta v^{\epsilon,w,k}f_{1,2} v^{\epsilon,w,k}_{,3}\\
\\
+(\delta v^{\epsilon,w,k}-f_1\delta v^{\epsilon,w,k}_{,1}-v^{\epsilon,w,k}\delta v^{\epsilon,w,k}_{,2}-\delta v^{\epsilon,w,k}v^{\epsilon,w,k}_{,2})f_{1,3}{\Big )}(s,y)\\
\\
\times G_{\epsilon}(t-s,x-y)dyds+\\
\\
\int_0^t\int_{{\mathbb R}^3}\left(S^{\epsilon,k+1}_w-S^{\epsilon,k}_w\right)(s,y)G_{\epsilon}(t-s,x-y)dyds,
\end{array}
\end{equation}
where
\begin{equation}\label{sksk}
\begin{array}{ll}
S^{\epsilon,k+1}_w-S^{\epsilon,k}_w\equiv \int_{{\mathbb R}^3} K_{,1}(x-y){\Big(} f_{1,1}^2+\left( v^{\epsilon ,w,k}_{,2}\right) ^2+\left( f_{1,1}+v^{\epsilon ,w,k}_{,2}\right) ^2+\\
\\
f_{1,2}v^{\epsilon,w,k}_{,1}+f_{1,3}w_{,1}+v^{\epsilon,w,k}_{,3}w_{,2}{\Big)}(y)dy{\Big)}\delta v^{\epsilon,w,k}_{,3}\\
\\
+\int_{{\mathbb R}^3} K_{,1}(x-y){\Big(} \left( v^{\epsilon ,w,k}_{,2}\right) ^2-\left( v^{\epsilon ,w,k-1}_{,2}\right) ^2+\left( f_{1,1}+v^{\epsilon ,w,k}_{,2}\right) ^2-\left( f_{1,1}+v^{\epsilon ,w,k-1}_{,2}\right) ^2\\
\\
f_{1,2}\delta v^{\epsilon,w,k}_{,1}+\delta v^{\epsilon,w,k}_{,3}w_{,2}{\Big)}(y)dy{\Big)} v^{\epsilon,w,k-1}_{,3}\\
\\
% 
% +\int_{{\mathbb R}^3} K_{,1}(x-y){\Big(} f_{1,1}^2+\left( v^{\epsilon ,k}_{,2}\right) ^2+\left( f_{1,1}+v^{\epsilon ,k}_{,2}\right) ^2+\\
% \\
% f_{1,2}v^{\epsilon,k}_{,1}+f_{1,3}I_{3,1}(f_{1,1},v^{\epsilon ,k}_{,2})+f_{2,3}I_{3,2}(f_{1,1},v^{\epsilon ,k}_{,2}){\Big)}(y)dy{\Big)} v^{\epsilon,k}_{,3}\\
% \\
-{\Big (} \int_{{\mathbb R}^3} K_{,2}(x-y){\Big (} \left( v^{\epsilon ,w,k}_{,2}\right) ^2-\left( v^{\epsilon ,w,k}_{,2}\right) ^2+\left( f_{1,1}+v^{\epsilon ,w,k}_{,2}\right) ^2-\left( f_{1,1}+v^{\epsilon ,w,k-1}_{,2}\right) ^2\\
\\
f_{1,2}\delta v^{\epsilon ,w,k}_{,1}+\delta v^{\epsilon ,w,k}_{,3}w_{,2}{\Big)}(y)dy{\Big)}f_{1,3}.
\end{array}
\end{equation}

%We have
% \begin{lem}
% For $p,m\geq 1$ and $x\neq 0$ we have
% \begin{equation}
% \int_{{\mathbb R}^n}\frac{C}{(x-y)^{\alpha}}|y|^{2m}dy\leq \frac{C}{|x|^{-n+\beta}}
% \end{equation}
% \end{lem}
% \begin{proof}
% We have for generic $C>0$
% \begin{equation}
% \begin{array}{ll}
% \int_{{\mathbb R}^n}\frac{C}{(x-y)^{2p}}|y|^{2m}dy\\
% \\
% \int_{{\mathbb R}^n,|y|<\frac{|x|}{2}}\frac{C}{(x-y)^{2p}}|y|^{2m}dy+\int_{{\mathbb R}^n,|y|>\frac{|x|}{2}}\frac{C}{(x-y)^{2p}}|y|^{2m}dy\\
% \\
% \leq int_{{\mathbb R}^n,|y|<\frac{|x|}{2}}\frac{C|y|}{(x-y)^{2p+2m}}|y|^{2m}dy+\int_{{\mathbb R}^n,|y|>\frac{|x|}{2}}\frac{C}{(x-y)^{2p}}|y|^{2m}dy
% \end{array}
% \end{equation}
% \end{proof}

Next we observe preservation of polynomial decay of certain order in the schemes. This observation holds also for the naive iteration scheme $v^{\epsilon,k}$, and it holds a fortiori for the iteration scheme $v^{\epsilon,k,w}$ 
\begin{lem}
 Polynomial decay of order $m$ of the functions $v^{\epsilon,k+1}$ and $\delta v^{\epsilon,k+1}$ and their derivatives of order $m$ is preserved if polynomial decay of order $m$ is given for the functions $v^{\epsilon,k}(t,.)\in C^m\cap H^m$, $\delta v^{\epsilon,k}(t,.)\in C^m\cap H^m$ (for all $t\geq 0$), and $f_1\in C^m\cap H^m$ and their derivatives of order $m$. An analogous statement holds for the iteration scheme $v^{\epsilon,w,k+1}$ for the parameterized Cauchy problem.
 \end{lem}
 \begin{proof}
 We prove the statement for the case $m=2$ which is essential for our purposes. We can prove the lemma with the original source terms $S^{\epsilon,k+1}$ for the original Cauchy problem which is is obtained by replacing $w_{,2}$ by $I_{3,2}(f_{1,1},v^{\epsilon ,k}_{,2})$ and $w_1$ by  $I_{3,1}(f_{1,1},v^{\epsilon ,k}_{,2})$ in $S^{\epsilon,k+1}_w$. The proof for the parmeterized Cauchy problem with the source terms $S^{\epsilon,k+1}_w$ is analogous and simpler. The generalisation of the following argument to higher order $m$ is also straightforward.
  We show that we have for $m\geq 1$
 \begin{equation}
 |\delta v^{\epsilon ,k+1}|\leq \frac{1}{|x|^{2m}},\mbox{ if }|x|\geq 1
 \end{equation}
and similarly for $v^{\epsilon k+1}$.
We have to estimate convolutions with the Gaussian $G_{\epsilon}$. The expressions for $\delta v^{\epsilon,k+1}$ and $v^{\epsilon,k+1}$ above are of the form
 \begin{equation}
 \int_{{\mathbb R}^3}h(y)G_{\epsilon}(x-y)dy.
 \end{equation}
where $h$ is some function which is a functional of $v^{\epsilon,k}(s,.)$ and $\delta v^{\epsilon ,k}(s,.)$ and $f_1$ are assumed to be in $H^2\cap C^2$, where the latter functions and their derivatives up to esond order are assumed to be of polynomial decay of order $m=2$. Furthermore we have the Gaussian factor $G_{\epsilon}(t-s,x-.)$ is . We split up the integral of the convolution into two parts where one part is the integral for $|y|\leq \frac{|x|}{2}$. On this domain
  we observe that the Gaussian has polynomial decay of any order. Indeed we have 
 \begin{equation}
  \begin{array}{ll}
{\big |}G_{\epsilon}(t-s,x-y){\big |}={\big |}\frac{1}{(2\sqrt{\epsilon \pi (t-s)})^n}\exp\left(-\frac{(x-y)^2}{4\epsilon (t-s) } \right){\big |}\\
\\
={\big |}(t-s)^{m-n/2}\frac{1}{(x-y)^{2m}}\left( \frac{(x-y)^2}{4\epsilon (t-s)}\right)^m \frac{1}{(2\sqrt{\epsilon \pi })^n}\exp\left(-\frac{(x-y)^2}{4\epsilon (t-s) } \right){\big |}\\
\\
\leq {\big |}(t-s)^{m-n/2}\frac{1}{(x-y)^{2m}}\left( \frac{(x-y)^2}{ (t-s)}\right)^m \frac{1}{(4\epsilon)^m}\frac{1}{(2\sqrt{\epsilon \pi })^n}\exp\left(-\frac{(x-y)^2}{4\epsilon (t-s) } \right){\big |}\\
\\
\leq {\big |}C(t-s)^{m-n/2}\frac{1}{(x-y)^{2m}} {\big |},
\end{array}
\end{equation}
where 
\begin{equation}
C:={\big |}\frac{1}{(4\epsilon)^m}\frac{1}{(2\sqrt{\epsilon \pi })^n}\left( \frac{(x-y)^2}{ (t-s)}\right)^m\exp\left(-\frac{(x-y)^2}{4\epsilon (t-s) } \right){\big |}>0
\end{equation}
is a constant depending on $\epsilon$, but finite for each $\epsilon >0$. 
On the complementary domain $|y|\geq\frac{|x|}{2}$ we need some properties of the integrand $h$.
Next 
we observe
\begin{equation}
\begin{array}{ll}
\int_{ \left\lbrace |y|\geq \frac{|x|}{2}\right\rbrace }\frac{C}{y^{2p}}G_{\epsilon}(x-y)dy\\
\\
\leq \int_{ \left\lbrace |y|\geq \frac{|x|}{2}\right\rbrace\&\left\lbrace |x-y|\leq 1\right\rbrace }\frac{C}{y^{2p}}G_{\epsilon}(x-y)dy\\
\\
+\int_{ \left\lbrace |y|\geq \frac{|x|}{2}\right\rbrace\&\left\lbrace |x-y|> 1,\right\rbrace }\frac{C}{y^{2p}}G_{\epsilon}(x-y)dy\\
\\
\leq \int_{\left\lbrace |y|\geq \frac{|x|}{2}\right\rbrace\& \left\lbrace |x-y|\leq 1\right\rbrace }\frac{C}{y^{2p}}{\big |}C(t-s)^{-1/2}\frac{1}{(x-y)^{2}} {\big |}dy\\
\\
+\int_{\left\lbrace |y|\geq \frac{|x|}{2}\right\rbrace\& \left\lbrace |x-y|> 1\right\rbrace }\frac{C}{y^{2p}}dy\leq \frac{2^{p-1}C^3}{|x|^{2p-1}}
\end{array}
\end{equation}
for some generic constant $C>0$. Hence the proof reduces to the observation that the integrand $h$ in the form it has in the representation of $\delta v^{k+1,\epsilon}$ is of polynomial decay. Since $v^{\epsilon ,k}$ and $\delta v^{\epsilon ,k}$ are of polynomial decay inductively this implies that $v^{\epsilon ,k+1}$ is of polynomial decay too. Now upon inspection of (\ref{deltavrep}) we observe that
\begin{equation}\label{deltavrep1}
\begin{array}{ll}
{\Big(}(-f_1+f_1f_{1,1}+v^{\epsilon,w,k}f_{1,2})\delta v^{\epsilon,w,k}_{,3}
+\delta v^{\epsilon,w,k}f_{1,2} v^{\epsilon,w,k}_{,3}\\
\\
+(\delta v^{\epsilon,w,k}-f_1\delta v^{\epsilon,w,k}_{,1}-v^{\epsilon,w,k}\delta v^{\epsilon,w,k}_{,2}-\delta v^{\epsilon,w,k}v^{\epsilon,w,k}_{,2})f_{1,3}{\Big )}
\end{array}
\end{equation}
is of polynomial decay of order $m=2$ by the inductive assumption of polynomial decay of order $m=2$ of the functions $v^{\epsilon,k}$, $\delta v^{\epsilon ,k}$, and $v^{\epsilon ,k-1}$, and of their first order derivatives. 
The integrand $\left( S^{\epsilon,k+1}-S^{\epsilon,k}\right)(s,.)$ deserves more detailed consideration since it is itself a convolution. Note that for $|y|\leq \frac{|x|}{2}$ we can use the Gaussian polynomial decay as above. Hence it is sufficient to estimate integrals of the form
\begin{equation}
\int_{\left\lbrace |y|\geq \frac{|x|}{2}\right\rbrace}
K_{,i}(y-z)g(z)dz
\end{equation}
where $g$ is a functional of $v^{\epsilon,k},v^{\epsilon,k-1}$ 
and $\delta v^{\epsilon,k}$ and their first order derivatives. Note that there are additional factors of the form $v^{\epsilon,k}_{,3}$ in the definition of (\ref{sksk}) but they have polynomial decay of order $m$ by assumption and products of functions of polynomial decay of order $m$ certainly have polynomial decay of order $m$. Hence we may neglect this factor.  Assuming 
\begin{equation}
|g(z)|\leq \frac{C}{z^{2m}}\mbox{ for } |z|\geq \frac{|x|}{4}
\end{equation}
for some generic $C>0$ we can use a similar argument as above and write
\begin{equation}
\begin{array}{ll}
{\big |}\int_{\left\lbrace |y|\geq \frac{|x|}{2}\right\rbrace}K_{,i}(y-z)g(z)dz{\big |}\\
\\
\leq {\big |}\int_{\left\lbrace |y|\geq \frac{|x|}{2}\right\rbrace \& \left\lbrace |z|\leq \frac{|y|}{2}\right\rbrace}K_{,i}(y-z)g(z)dz{\big |}\\
\\
+{\big |}\int_{\left\lbrace |y|\geq \frac{|x|}{2}\right\rbrace \& \left\lbrace |z|> \frac{|y|}{2}\right\rbrace}K_{,i}(y-z)g(z)dz{\big |}\\
\\
\leq {\big |}\int_{\left\lbrace |y|\geq \frac{|x|}{2}\right\rbrace \& \left\lbrace |z|\leq \frac{|y|}{2}\right\rbrace}\frac{\partial^{2m-2}}{\partial x_i^{2m-2}}K_{,i}(y-z)\frac{C}{z^2}dz{\big |}\\
\\
+{\big |}\int_{\left\lbrace |y|\geq \frac{|x|}{2}\right\rbrace \& \left\lbrace |z|> \frac{|y|}{2}\right\rbrace}K_{,i}(y-z)\frac{C}{y^{2m}}dz{\big |}\in O\left(\frac{C}{|x|^{2m-1}}. \right) 
\end{array}
\end{equation}
It remains to verify that $g$ as implicitly defined in the expression for $S^{k+1}-S^k$ satisfies the requirements needed for $g$ in the last step. Now in (\ref{sksk}) all summands of the form $f_{1,1}^2,\left( v^{\epsilon ,k}_{,2}\right)^2,\left( f_{1,1}+v^{\epsilon ,k}_{,2}\right) ^2,
f_{1,2}v^{\epsilon,k}_{,1}$ are certainly of polynomial decay of order $m$ as $f_{1,2}$ is this by assumption and $v^{\epsilon,k}$ and its derivatives up to order $m=2$ by inductive assumption. 

 Moreover, terms of the form $f_{1,3}I_{3,1}(f_{1,1},v^{\epsilon ,k}_{,2}), v^{\epsilon,k}_{,3}I_{3,2}(f_{1,1},v^{\epsilon ,k}_{,2})$ are of polynomial decay of order $m$ since the integral forms are $I_{3,1}(f_{1,1},v^{\epsilon ,k}_{,2})$ and  are bounded $I_{3,2}(f_{1,1},v^{\epsilon ,k}_{,2})$ and their factors $f_{1,3}$ is of polynomial decay of order $m$ by assumption and $v^{\epsilon,k}_{,3}$ is of polynomial decay by inductive assumption. 
This shows that the considerations above apply for the first integral in the expression for (\ref{sksk}). Similarly for the other integrals in (\ref{sksk}). We have shown that $\delta v^{\epsilon,k}$ is of polynomial decay of order $m=2$ 
if $v^{\epsilon,k}$ and $\delta v^{\epsilon,k}$ and their derivatives up to order $m$ (and inductively all predecessors in the order of stage $k$) are of polynomial decay of order $m$.  
For the first order derivatives we have still locally integrable bounds of the first derivatives of the Gaussian such that the argument above can be repeated with standard a priori estimates for the first derivative of the Gaussian. 
For the second derivatives we shift over one derivative of the Laplacian kernel $K_{,i}$ to the integrand of form $h$ or $g$ above. These integrands involve derivatives of $f, v^{\epsilon ,k}, \delta v^{\epsilon,k}$ of first order. Hence the inductive assumption for order $m=2$ is still satisfied if these terms are differentiated. Furthermore we shift one derivative from the Gaussian to the Laplacian kernel $K$ and then repeat the argument above concluding that second order derivatives of $v^{\epsilon ,k+1}$ and $\delta v^{\epsilon ,k+1}$ also have polynomial decay of order $m=2$.   
\end{proof}

Next we prove a contraction property for the parameterized Cauchy problem which implies the existence of a classical solution $v^{\epsilon,w}$ with $v^{\epsilon,w}(t,.)\in H^2\cap C^2$ for all $t\geq 0$. Again $w$ is a parameter in $C^2_b$.
\begin{lem}
Assume that $v^{\epsilon,w,k}$ and $v^{\epsilon,w, k-1}$ and their derivatives up to second order are of polynomial decay of order $m=2$.
For any choice
\begin{equation}
\begin{array}{ll}
C\geq 4(2+(2n+1)6C_0)^2C_1
+4{\big (}(2n+1)(C_2+(2n+1)C_3)(5C_0^2+2C_0C_4)\\
\\
+2(C_2+(2n+1)C_3)(2n+1)C_08(2n+1)C_0){\big)}
\end{array}\end{equation}
we have
\begin{equation}
{\big |} \delta v^{\epsilon ,w,k+1}{\big |}^{\mbox{exp}}_{H^2}\leq \frac{1}{2}{\big |} \delta v^{\epsilon,w ,k}{\big |}^{\mbox{exp},C}_{H^2}.
\end{equation}
If $v^{\epsilon,w,k}$ and $v^{\epsilon,w, k-1}$ and their derivatives up to order $m$ are of polynomial decay of order $m$ then a similar result can be obtained for an analogous norm  ${\big |} \delta v^{\epsilon,w ,k}{\big |}^{\mbox{exp},C}_{H^m}$.
\end{lem}

\begin{proof}
We consider the case $m=2$. The generalisation to $m>2$ is straightforward.
We may rewite (\ref{deltavrep}) and write
\begin{equation}\label{delta2}
\begin{array}{ll}
\delta v^{\epsilon,w,k+1}=\int_0^t\int_{{\mathbb R}^3}{\Big(}(-f_1+f_1f_{1,1}+v^{\epsilon,w,k}f_{1,2})\delta v^{\epsilon,w,k}_{,3}
+\delta v^{\epsilon,w,k}f_{1,2} v^{\epsilon,w,k}_{,3}\\
\\
+(\delta v^{\epsilon,w,k}-f_1\delta v^{\epsilon,w,k}_{,1}-v^{\epsilon,w,k}\delta v^{\epsilon,w,k}_{,2}-\delta v^{\epsilon,w,k}v^{\epsilon,w,k}_{,2})f_{1,3}{\Big )}(s,y)dyds +\\
\\
\int_0^t\int_{{\mathbb R}^3}\left( S^{\epsilon,k+1}_w-S^{\epsilon,k}_w\right)(s,y)G_{\epsilon}(t-s,x-y)dyds\\
\\
=:\delta v^{\epsilon,w,k+1,1}+\delta v^{\epsilon,w,k+1,2},
\end{array}
\end{equation}
and estimate the first integra term on the right side of (\ref{delta2}) separately, since it does not contain a Laplacian kernel, and can be treated differently therefore. For the first and second order derivatives of $\delta v^{\epsilon,w,k+1,1}$ we consider the representations
\begin{equation}
\begin{array}{ll}
\delta v^{\epsilon,w,k+1,1}_{,i}=
\int_0^t\int_{{\mathbb R}^3}{\Big(}(-f_1+f_1f_{1,1}+v^{\epsilon,w,k}f_{1,2})\delta v^{\epsilon,w,k}_{,3}
+\delta v^{\epsilon,w,k}f_{1,2} v^{\epsilon,w,k}_{,3}\\
\\
+(\delta v^{\epsilon,w,k}-f_1\delta v^{\epsilon,w,k}_{,1}-v^{\epsilon,w,k}\delta v^{\epsilon,w,k}_{,2}-\delta v^{\epsilon,w,k}v^{\epsilon,w,k}_{,2})f_{1,3}{\Big )}(s,y)\\
\\
\times G_{\epsilon,i}(t-s,x-y)dyds,
\end{array}
\end{equation}
and
\begin{equation}
\begin{array}{ll}
\delta v^{\epsilon,w,k+1,1}_{,i,j}=
\int_0^t\int_{{\mathbb R}^3}{\Big(}(-f_1+f_1f_{1,1}+v^{\epsilon,w,k}f_{1,2})\delta v^{\epsilon,w,k}_{,3}
+\delta v^{\epsilon,w,k}f_{1,2} v^{\epsilon,w,k}_{,3}\\
\\
+(\delta v^{\epsilon,w,k}-f_1\delta v^{\epsilon,w,k}_{,1}-v^{\epsilon,w,k}\delta v^{\epsilon,w,k}_{,2}-\delta v^{\epsilon,w,k}v^{\epsilon,w,k}_{,2})f_{1,3}{\Big )}_{,i}(s,y)\\
\\
\times G_{\epsilon,j}(t-s,x-y)dyds.
\end{array}
\end{equation}
For dimension $n=3$ let
\begin{equation}
\begin{array}{ll}
C:=4(2+(2n+1)6C_0)^2C_1
+4{\big (}(2n+1)(C_2+(2n+1)C_3)(5C_0^2+2C_0C_4)\\
\\
+2(C_2+(2n+1)C_3)(2n+1)C_08(2n+1)C_0){\big)}
\end{array}
\end{equation}
(the reason with this definition will become apparent later). We assume inductively that
\begin{equation}
\sup_{t\geq 0,x\in {\mathbb R}^n}\left( |v^{\epsilon ,w,k}(t,x)|
+\sum_{i=1}^n|v^{\epsilon ,w,k}_{,i}(t,x)|+\sum_{i=1,j}^n|v^{\epsilon ,w,k}_{,i,j}(t,x)|\right) \leq C_0,
\end{equation}
along with $\sup_{t\geq 0}|v^{\epsilon,w,k}(t,.)|_{H^2}\leq C_0$ for the same generic constant $C_0$. In order to avoid inflation of the number of constants we assume w.l.o.g. the same upper bound for the relevant norms of the parameter function $w$, i.e.,
\begin{equation}
\sup_{t\geq 0,x\in {\mathbb R}^n}\left( |w(t,x)|
+\sum_{i=1}^n|w_{,i}(t,x)|+\sum_{i=1,j}^n|w_{,i,j}(t,x)|\right) \leq C_0,
\end{equation}
along with $\sup_{\tau\geq 0}|w(t,.)|_{H^2}\leq C_0$.
For the present we shall take this as in induction hypothesis at stage $k$ (for $k=0$ recall that we defined $v^{\epsilon ,0}=-f_1$- hence, at stage $k=0$ the induction hypothesis is satisfied). Estimates in $L^p$ of Gaussians like $G_{\epsilon}(t-s,.)$ or its first order derivatives can be obtained by breaking up integrals. 
Consider a function $\phi_1\in C^{\infty}\left(B_1(0))\right)$, i.e. with support in $B_1(0)$, and with $\phi_1(x)=1$ for $|x|\leq 0.5$. Note that $\phi_1$ and $1-\phi_1$ build a partition of unity on ${\mathbb R}^3$. We refer to \cite{KB2}, but this can be found in standard texts for sure. In addition we may use a generalized Young inequality
\begin{equation}
\begin{array}{ll}
|f\ast g|_{L^r}\leq |f|_{L^p}|g|_{L^q}, \mbox{ for all } 1\leq p,q,r,\leq \infty \mbox{ where }
\frac{1}{p}+\frac{1}{q}=1+\frac{1}{r},
\end{array}
\end{equation}
For $t>s$ we have $|G_{\epsilon}(t-s,.)|_{L^1}\leq C_1< \infty$ and $|G_{\epsilon ,i}(t-s,.)|_{L^1}\leq C_1< \infty$ for all $1\leq i\leq n$ and some $C_1>0$ which is even independent of the size of $t-s$. We get the (generous) estimate
\begin{equation}
\begin{array}{ll}
\exp(-Ct)|\delta v^{\epsilon,w,k+1,1}(t,.)|_{H^2}\\
\\
\leq \exp(-Ct)\int_0^t (2+(2n+1)C_0)^2C_1|\delta v^{\epsilon,w ,k}(s,.)|_{H^2}ds\\
\\
\leq \exp(-Ct)\int_0^t\exp(Cs)(2+(2n+1)6C_0)^2 C_1|\delta v^{\epsilon ,w,k}|^{\mbox{exp},C}_{H^2}ds\\
\\
\leq \exp(-Ct)\frac{(2+(2n+1)6C_0)^2C_1}{C}\left( \exp(Ct)-1\right) |\delta v^{\epsilon ,w,k}|^{\mbox{exp},C}_{H^2}\\
\\
\leq \frac{(2+(2n+1)6C_0)^2C_1}{C} |\delta v^{\epsilon ,w,k}|^{\mbox{exp},C}_{H^2}.
\end{array}
\end{equation}
Hence, by choice of $C>0$ we have
\begin{equation}
|\delta v^{\epsilon ,w,k+1,1 }|^{\mbox{exp},C}_{H^2}\leq \frac{1}{4} |\delta v^{\epsilon ,w,k}(s,.)|^{\mbox{exp},C}_{H^2}.
\end{equation}
Finally we do the estimate for $|\delta v^{\epsilon ,w,k,3 }|^{\mbox{exp},C}_{H^2}$, i.e., the $H^2$-estimate for
\begin{equation}
\int_0^t\int_{{\mathbb R}^3}\left( S^{\epsilon,k+1}_w-S^{\epsilon,k}_w\right)(s,y)G_{\epsilon}(t-s,x-y)dyds
\end{equation}
Using the generalized Young inequality this reduces to a $H^1$-estimate for
\begin{equation}
 S^{\epsilon,k+1}-S^{\epsilon,k}
\end{equation}
The latter term involves a convolution with the Laplacian kernel and we use the partition of unity in ${\mathbb R}^3$ with $\phi_1$ (support $B_1(0)$) and $(1-\phi_1)$ as defined above (cf. also \cite{KB2} for similar considerations).
For all $1\leq i\leq n$ we have
\begin{equation}
\int_{{\mathbb R}^3}|\phi_1K_{,i}(z)|d^3z =\int_{{\mathbb R}^3}|\phi_1(z)|{\Big |}\frac{z_i}{|z|^3}{\Big |}d^3z\sim \int_{{\mathbb R}^3}|\phi_1(z)|{\Big |}\frac{r}{|r|^3}{\Big |}r^2dr= C
\end{equation}
for some constant $C>0$. Hence, $\phi_1K_{,i}\in L^1$, and 
in order finish our prove of a contractive property we consider again three different parts of the expression for $S^{k+1}_w-S^k_w=S^{\epsilon,k+1,1}_w-S^{\epsilon,k,1}_w+S^{\epsilon,k+1,2}_w-S^{\epsilon,k,2}_w+S^{\epsilon,k+1,3}_w-S^{\epsilon,k,3}_w$ considered above. All estimates can be done by splitting up the integrals and considering the sums involving $\phi_1K_{,i}\in L^1$ and $(1-\phi_1)K_{,i}\in L^2\cap L^{\infty}$ separately. We have finite constants
\begin{equation}
C_2:=|\phi_1K_{,i}|_{L^1}
\end{equation}
and
\begin{equation}
C_3:=|(1-\phi_1)K_{,i}|_{L^2}+|(1-\phi_1)K_{,i}|_{L^{\infty}}+\sum_{j=1}^n|\left( (1-\phi_1)K_{,i}\right)_{,j}|_{L^{\infty}}
\end{equation}
and we estimate the summands involving the (outside) truncated kernel $\phi_1K_{,i}$ by the generalized Young inequality and the summands for the kernel (truncated inside) $(1-\phi_1)K_{,i}$ with some basic techniques of the Fourier transform. In addition to the preceding inductive assumptions concerning the bounds for $v^{\epsilon ,k}$ and $f_1$ we mention our assumption on $w$. We have
\begin{equation}
\begin{array}{ll}
\sup_{t\geq 0,x\in {\mathbb R}^n}2{\Big(}|w(t,x)|+|\sum_i w_{,i})(t,x)|\\
\\
+|\sum_{i,j}w_{,i,j})(t,x)|{\Big )}\leq C_4
\end{array}
\end{equation}
 for all $t\geq 0$ and some constant $C_4$. The first term $S^{\epsilon,k+1,1}_w-S^{\epsilon,k,1}_w$ can be estimated by referring to the suprema.
We remark that for bounded $C^1$-functions $g$ and functions $\delta v^{\epsilon,k}$ with bounded first order derivatives we have $H^1$-estimates for expressions of the form
\begin{equation}
\begin{array}{ll}
|g\delta v^{\epsilon,k}|_{H^1}=|g\delta v^{\epsilon,w,k}_{,j}|_{L^2}+\sum_{i=1}^n|(g\delta v^{\epsilon,w,k})_{,j,i}|_{L^2}\\
\\
\leq \sqrt{\int_{{\mathbb R}^3}g(y)^2\delta v^{\epsilon,w,k}_{,j}(y)^2dy}
+\sum_{i=1}^n|g_{,i}\delta v^{\epsilon,w,k}_{,j}|_{L^2}
+\sum_{i=1}^n|g\delta v^{\epsilon,w,k}_{,j,i}|_{L^2}\\
\\
\leq (2n+1)C|\delta v^{\epsilon,w,k}|_{H^1}
\end{array}
\end{equation}
for a constant  $C>0$ which is an upper bound for $|g|_{L^{\infty}}$ and $|g_{,i}|_{L^{\infty}}$, i.e., all first order derivatives.
In the following we fix time $t\geq 0$ arbitrarily and suppress the notation of time.
This way we get  
\begin{equation}
\begin{array}{ll}
{\big|}S^{\epsilon,k+1,1}_w-S^{\epsilon,k,1}_w{\big |}_{H^1}\equiv {\big |}\int_{{\mathbb R}^3} K_{,1}(x-y){\Big(} f_{1,1}^2+\left( v^{\epsilon ,w,k}_{,2}\right) ^2+\left( f_{1,1}+v^{\epsilon ,w,k}_{,2}\right) ^2+\\
\\
f_{1,2}v^{\epsilon,w,k}_{,1}+f_{1,3}w_{,1}+f_{2,3}w_{,2}{\Big)}(y)dy{\Big)}\delta v^{\epsilon,w,k}_{,3}{\big |}_{H^1}\\
\\
\leq C_2{\Big|} {\Big (}f_{1,1}^2+\left( v^{\epsilon ,w,k}_{,2}\right) ^2+\left( f_{1,1}+v^{\epsilon ,k}_{,2}\right) ^2+
f_{1,2}v^{\epsilon,w,k}_{,1}+f_{1,3}w_{,1}+f_{2,3}w_{,2}{\Big)}\delta v^{\epsilon,w,k}_{,3}
{\Big|}_{H^1}\\
\\
+(2n+1)C_3{\Big|} {\Big (}f_{1,1}^2+\left( v^{\epsilon ,w,k}_{,2}\right) ^2+\left( f_{1,1}+v^{\epsilon ,w,k}_{,2}\right) ^2+
f_{1,2}v^{\epsilon,w,k}_{,1}+f_{1,3}w_{,1}+f_{2,3}w_{,2}{\Big)}\delta v^{\epsilon,w,k}_{,3}
{\Big|}_{H^1}\\
\\
\leq (2n+1)(C_2+(2n+1)C_3)(5C_0^2+2C_0C_4)|\delta v^{\epsilon,w,k}_{,3}|_{H^1}\\
\\
\leq (2n+1)(C_2+(2n+1)C_3)(5C_0^2+2C_0C_4)|\delta v^{\epsilon,w,k}|_{H^2},
\end{array}
\end{equation}
where we use the $H^{1,\infty}$ bound for the factor $(1-\phi)K_{,i}$ and its first order derivative.
For the second term we have
\begin{equation}
\begin{array}{ll}
{\big |}S^{\epsilon,k+1,2}_w-S^{\epsilon,k,2}_w{\big |}_{H^1}= \\
\\
{\big |}\int_{{\mathbb R}^3} K_{,1}(x-y){\Big(} \left( v^{\epsilon ,w,k}_{,2}\right) ^2-\left( v^{\epsilon ,w,k-1}_{,2}\right) ^2+\left( f_{1,1}+v^{\epsilon ,w,k}_{,2}\right) ^2\\
\\
-\left( f_{1,1}+v^{\epsilon ,w,k-1}_{,2}\right) ^2
+f_{1,2}\delta v^{\epsilon,w,k}_{,1}+\delta v^{\epsilon,w,k}_{,3}w_{,2}{\Big)}(y)dy{\Big)} v^{\epsilon,k-1}_{,3}{\big |}_{H^1}\\
\\
\leq {\big |}\int_{{\mathbb R}^3} (\phi_1K_{,1})(x-y){\Big(} \left( v^{\epsilon ,w,k}_{,2}+v^{\epsilon ,w,k-1}_{,2}\right)\delta v^{\epsilon ,w,k}_{,2}+\\
\\
\left(2 f_{1,1}+v^{\epsilon ,w,k}_{,2}+v^{\epsilon ,w,k-1}_{,2}\right) \delta v^{\epsilon,w,k}_{,2}+
f_{1,2}\delta v^{\epsilon,w,k}_{,1}+\delta v^{\epsilon,w,k}_{,3}w_{,2}{\Big)}(y)dy{\Big)} v^{\epsilon,w,k-1}_{,3}{\big |}_{H^1}\\
\\
+{\big |}\int_{{\mathbb R}^3} (1-\phi_1)K_{,1})(x-y){\Big(} \left( v^{\epsilon ,w,k}_{,2}+v^{\epsilon ,w,k-1}_{,2}\right)\delta v^{\epsilon ,w,k}_{,2}+\\
\\
\left(2 f_{1,1}+v^{\epsilon ,w,k}_{,2}+v^{\epsilon ,w,k-1}_{,2}\right) \delta v^{\epsilon,w,k}_{,2}+
f_{1,2}\delta v^{\epsilon,w,k}_{,1}+\delta v^{\epsilon,w,k}_{,3}w_{,2}{\Big)}(y)dy{\Big)} v^{\epsilon,w,k-1}_{,3}{\big |}_{H^1}\\
\\
\leq (C_2+(2n+1)C_3){\big |}\int_{{\mathbb R}^3} {\Big(} \left( v^{\epsilon ,w,k}_{,2}+v^{\epsilon ,w,k-1}_{,2}\right)\delta v^{\epsilon ,w,k}_{,2}+\\
\\
\left(2 f_{1,1}+v^{\epsilon ,w,k}_{,2}+v^{\epsilon ,w,k-1}_{,2}\right) \delta v^{\epsilon,w,k}_{,2}+
f_{1,2}\delta v^{\epsilon,w,k}_{,1}+\delta v^{\epsilon,w,k}_{,3}w_{,2}{\Big)}{\Big)} v^{\epsilon,w,k-1}_{,3}{\big |}_{H^1}\\
\\
\leq (C_2+(2n+1)C_3)(2n+1)C_08(2n+1)C_0)|\delta v^{\epsilon,w k}(t,.)|_{H^2},
\end{array}
\end{equation}
uniformly for all $t\geq 0$, and where we suppressed the notation of the parameter $t\geq 0$ for all terms except the last one.  
The first term on the right side of the latter equation can be estimated using the generalized Young inequality. 
Note that the third term differs from the second term only with respect to the kernel function $K_{,2}$ instead of $K_{,1}$. Hence, we have the same estimate for
\begin{equation}
\begin{array}{ll}
S^{\epsilon,k+1,3}-S^{\epsilon,k,3}\equiv\\
\\
-{\Big (} \int_{{\mathbb R}^3} K_{,2}(x-y){\Big (} \left( v^{\epsilon ,k}_{,2}\right) ^2-\left( v^{\epsilon ,k}_{,2}\right) ^2+\left( f_{1,1}+v^{\epsilon ,k}_{,2}\right) ^2-\left( f_{1,1}+v^{\epsilon ,k-1}_{,2}\right) ^2\\
\\
f_{1,2}\delta v^{\epsilon ,k}_{2,1}+f_{1,3}I_{3,1}(f_{1,1},\delta v^{\epsilon ,k}_{2,2})+\delta v^{\epsilon ,k}_{,3}I_{3,2}(f_{1,1},v^{\epsilon ,k}_{,2})+v^{\epsilon ,k-1}_{,3}I_{3,2}(f_{1,1},\delta v^{\epsilon ,k}_{,2}){\Big)}(y)dy{\Big)}f_{1,3}.
\end{array}
\end{equation}
Summing up the estimates fr the source terms of the form $S^{\epsilon,k+1}_w-S^{\epsilon,k}_w$ we have 
\begin{equation}
\begin{array}{ll}
{\big |}S^{\epsilon,k+1}_w-S^{\epsilon,k}_w{\big |}_{H^1}= \\
\\
\leq (2+(2n+1)6C_0)^2C_1+{\big (}(2n+1)(C_2+(2n+1)C_3)(5C_0^2+2C_0C_4)\\
\\
+2(C_2+(2n+1)C_3)(2n+1)C_08(2n+1)C_0){\big)}|\delta v^{\epsilon,w k}(t,.)|_{H^2}.
\end{array}
\end{equation}
Hence, we observe that the choice
\begin{equation}
\begin{array}{ll}
C:=4(2+(2n+1)6C_0)^2C_1+4{\big (}(2n+1)(C_2+(2n+1)C_3)(5C_0^2+2C_0C_4)\\
\\
+2(C_2+(2n+1)C_3)(2n+1)C_08(2n+1)C_0){\big)}
\end{array}
\end{equation}
is sufficient in order to get the $|.|^{\mbox{exp},C}_{H^2}$ with factor $\frac{1}{2}$.
\end{proof}

We have obtained a solution $v^w$ of the parameterized Cauchy problem. Next a rather simple argument (using similar techniques as above) leads to a stability with respect to $w$.
\begin{lem}
The parametrized solution $v^{\epsilon,w}$ is stable with respect to $w(t,.)\in  C^2_b$ (uniformly w.r.t. $t$), i.e., for each $\epsilon >0$ we have
\begin{equation}
|v^{\epsilon,w}(t,.)-v^{\epsilon,w'}(t,.)|_{H^2}\leq C|w(t,.)-w'(t,.)|_{C^2}
\end{equation}
for all $t\geq 0$ and some constant $C>0$.
\end{lem}

Now for each $w\in  C^2_b$ we have obtained a solution $v^{\epsilon}_w$ of the parameterized Cauchy problem which has the representation
\begin{equation}\label{Cauchyparameter}
\begin{array}{ll}
v^{\epsilon,w}(t,x)=\int_0^t\int_{{\mathbb R}^3}{\Big(}(-f_1+f_1f_{1,1}+v^{\epsilon,w}f_{1,2})v^{\epsilon,w}_{,3}\\
\\
+(v^{\epsilon,w}-f_1v^{\epsilon,w,k}_{,1}-v^{\epsilon,w}v^{\epsilon,w}_{,2})f_{1,3}{\Big )}\\
\\
\times \frac{1}{(2\sqrt{\epsilon \pi (t-s)})^n}\exp\left(-\frac{(x-y)^2}{4\epsilon (t-s) } \right) dyds\\
\\
+\int_0^t{\Big (}\int_{{\mathbb R}^3} K_{,1}(z-y){\Big(} f_{1,1}^2+\left( v^{\epsilon,w}_{,2}\right) ^2+\left( f_{1,1}+v^{\epsilon,w}_{,2}\right) ^2+\\
\\
f_{1,2}v^{\epsilon,w}_{,1}+f_{1,3}w_{,1}+v^{\epsilon,w}_{,3}w_{,2}{\Big)}(y)dy{\Big)}v^{\epsilon,w}_{,3}{\Big )}(s,z)\\
\\
\times \frac{1}{(2\sqrt{\epsilon \pi (t-s)})^n}\exp\left(-\frac{(x-y)^2}{4\epsilon (t-s) } \right)dzds\\
\\
-\int_0^t{\Big (} \int_{{\mathbb R}^3} K_{,2}(x-y){\Big (} f_{1,1}^2+\left( v^{\epsilon,w}_{,2}\right) ^2+\left( f_{1,1}+v^{\epsilon,w}_{,2}\right) ^2+\\
\\
f_{1,2}v^{\epsilon,w}_{2,1}+f_{1,3}w_{,1}+v^{\epsilon,w}_{,3}
w_{,2}{\Big)}(y)dyf_{1,3}{\Big)}(s,z)\\
\\
\times \frac{1}{(2\sqrt{\epsilon \pi (t-s)})^n}\exp\left(-\frac{(x-y)^2}{4\epsilon (t-s) } \right)dzds.
\end{array}
\end{equation}
The divergence property of the vector field $(f_1,v^{\epsilon,w},w)$ leads to
\begin{equation}\label{Cauchydivergence}
\begin{array}{ll}
-w_{,3}(t,x)=f_{1,1}(t,x)\\
\\
+\int_0^t\int_{{\mathbb R}^3}{\Big(}(-f_1+f_1f_{1,1}+v^{\epsilon,w}f_{1,2})v^{\epsilon,w}_{,3}\\
\\
+(v^{\epsilon,w,k}-f_1v^{\epsilon,w}_{,1}-v^{\epsilon,w,k}v^{\epsilon,w}_{,2})f_{1,3}{\Big )} G_{\epsilon,2}(t-s,x-y) dyds\\
\\
+\int_0^t{\Big (}\int_{{\mathbb R}^3} K_{,1}(z-y){\Big(} f_{1,1}^2+\left( v^{\epsilon,w}_{,2}\right) ^2+\left( f_{1,1}+v^{\epsilon,w}_{,2}\right) ^2+\\
\\
f_{1,2}v^{\epsilon,w}_{,1}+f_{1,3}w_{,1}+v^{\epsilon,w}_{,3}w_{,2}{\Big)}(y)dy{\Big)}v^{\epsilon,w}_{,3}{\Big )}(s,z) G_{\epsilon,2}(t-s,x-z)dzds\\
\\
-\int_0^t{\Big (} \int_{{\mathbb R}^3} K_{,2}(x-y){\Big (} f_{1,1}^2+\left( v^{\epsilon,w}_{,2}\right) ^2+\left( f_{1,1}+v^{\epsilon,w}_{,2}\right) ^2+\\
\\
f_{1,2}v^{\epsilon,w}_{2,1}+f_{1,3}w_{,1}+v^{\epsilon,w}_{,3}
w_{,2}{\Big)}(y)dyf_{1,3}{\Big)}(s,z) G_{\epsilon,2}(t-s,x-z)dzds.
\end{array}
\end{equation}
We can apply that $v^{\epsilon,w}$ is stable with respect to $w$. A simple iteration scheme and contraction estimate leads to the divergence free solution $(f_1,v^{\epsilon,w},w)$ of the Cauchy problem for each $\epsilon>0$ with $w_{,3}\in C^2\cap H^2$. However we cannot ensure the $L^2$-finiteness of the integral of $w_{,3}$ with respect to the third variable over the whole space and therefore for the velocity $v_3$ and $f_3$.
Applying this result we get
\begin{equation}
v^{\epsilon}:=v^{\epsilon ,0}+\sum_{k=1}^{\infty}\delta v^{\epsilon ,k}\in H^2_{\mbox{exp},C}.
\end{equation}
Hence for each $t\geq 0$ we have
\begin{equation}
v^{\epsilon}(t,.)\in H^2\subset C^{\alpha}
\end{equation}
for some $\alpha\in (0,1)$ and uniformly with respect to time $t$. Hence we have the representation 
 \begin{equation}\label{fundrep*}
\begin{array}{ll}
v^{\epsilon}(t,x)=\int_0^t\int_{{\mathbb R}^3}{\Big(}(-f_1+f_1f_{1,1}+v^{\epsilon}f_{1,2})v^{\epsilon}_{,3}\\
\\
+(v^{\epsilon,w,k}-f_1v^{\epsilon,w,k}_{,1}-v^{\epsilon}v^{\epsilon}_{,2})f_{1,3}{\Big )}G_{\epsilon}(t-s;x-y)dyds\\
\\
+\int_0^t{\Big (}\int_{{\mathbb R}^3} K_{,1}(z-y){\Big(} f_{1,1}^2+\left( v^{\epsilon}_{,2}\right) ^2+\left( f_{1,1}+v^{\epsilon}_{,2}\right) ^2+\\
\\
f_{1,2}v^{\epsilon}_{,1}+f_{1,3}I_{3,1}(f_{1,1},v^{\epsilon}_{,2})+f_{2,3}I_{3,2}(f_{1,1},v^{\epsilon}_{,2}){\Big)}(y)dy{\Big)}v^{\epsilon}_{,3}{\Big )}(s,z)\\
\\
\times G_{\epsilon}(t-s;x-z)dzds\\
\\
-\int_0^t{\Big (} \int_{{\mathbb R}^3} K_{,2}(x-y){\Big (} f_{1,1}^2+\left( v^{\epsilon}_{,2}\right) ^2+\left( f_{1,1}+v^{\epsilon}_{,2}\right) ^2+\\
\\
f_{1,2}v^{\epsilon}_{2,1}+f_{1,3}I_{3,1}(f_{1,1},v^{\epsilon}_{2,2})+v^{\epsilon}_{,3}I_{3,2}(f_{1,1},v^{\epsilon}_{,2}){\Big)}(y)dyf_{1,3}{\Big)}(s,z)\\
\\
\times G_{\epsilon}(t-s;x-z)dzds
\end{array}
\end{equation}
  of the solution $v^{\epsilon}$ of the Cauchy problem (\ref{Cauchyepsilon}). Spatial regularity can be obtained from this expression straightforwardly then. Furthermore, the argument above can be repeated for any order $m\geq 2$ of regularity $f_1\in C^m\cap H^m$ yielding $v^{\epsilon}(t,.),w(t,.)\in C^m\cap H^m$ uniformly with respect to $t$ and with polynomial decay of order $m$. Furthermore, using standard a priori estimates for the Gaussian as in the lemma on polynomial decay above we observe from (\ref{fundrep*}) we see that $|v^{\epsilon}(t,.)|_{H^2}\leq C$ for some constant $C>0$ which is independent of $\epsilon$. Furthermore you easily observe that this holds also for the solution $v^{\epsilon}_{\infty}:=\lim_{t\uparrow \infty}v^{\epsilon}(t,.)$ of the related elliptic problem.
% The essential step in order to show that the scheme is contractive is the following
% \begin{lem}
% \begin{equation}
% {\Bigg |}\int_{{\mathbb R}^3}K_{,i}(x-y)\sum_{j,k=1}^3 \left( g_{j,k}(y)g_{k,j}(y)-h_{j,k}(y)h_{k,j}(y)\right)dy{\Bigg |}_{H^s}\leq C\sum_{i=1}^3|g_i-h_i|_{H^s}  
% \end{equation}
% \end{lem}
% 
% 
% Then we use a Schauder type estimate of the form
% \begin{thm}
% Let
% \begin{equation}
% \Delta u=f,~\mbox{on}~\Omega,~\mbox{ and }~u|_{\partial \Omega}=g.
% \end{equation}
% Then
% \begin{equation}
% |u|_{k+2,\delta}\leq N\left(|Lu|_{k+\delta}+|g|_{k+2+\delta} \right) 
% \end{equation}
% for some constant $N$ depending only on properties of the domain.
% \end{thm}
 
\footnotetext[1]{\texttt{{kampen@mathalgorithm.de}, {kampen@wias-berlin.de}}.}

\end{document}